\documentclass[12pt,reqno,a4paper]{amsart}
\usepackage[square,numbers]{natbib}
\usepackage{amsmath,mathrsfs}
\usepackage{amssymb}
\usepackage{calligra}
\usepackage{dsfont}
\usepackage[english]{cleveref}
\usepackage{bbm}
\usepackage{enumitem}

\usepackage{color}
\usepackage{mathtools}
\usepackage{todonotes}

\newcommand\NoBlackBoxes{\global\overfullrule0pt}
\NoBlackBoxes

\theoremstyle{plain}\newtheorem{thm}{Theorem}[section]
\theoremstyle{plain}
\theoremstyle{definition}
\theoremstyle{definition}
\theoremstyle{plain}\newtheorem{defi}[thm]{Definition}
\theoremstyle{plain}\newtheorem{lem}[thm]{Lemma}
\theoremstyle{plain}\newtheorem{prop}[thm]{Proposition}
\theoremstyle{definition}

\numberwithin{equation}{section}

\newcommand{\E}{{\mathbb{E}}}
\newcommand{\N}{{\mathbb{N}}}
\newcommand{\R}{{\mathbb{R}}}

\renewcommand{\P}{{\mathbb{P}}}
\newcommand{\ind}{\operatorname{1}}

\renewcommand{\ind}[1]{\mathbbm{1}_{\left\{#1\right\}}}
\newcommand{\WK}[1]{\mathbb{P}\left(#1\right)}
\newcommand{\EW}[1]{\mathbb{E}\left[#1\right]}

\newcommand{\ERG}{Erd\H{o}s-R\'enyi random graph}

\newcommand{\1}{\mathbbm{1}}

\let\eps\varepsilon

\newcommand{\nconv}{\xrightarrow[]{n\to\infty}}
\newcommand{\Pconv}{\xrightarrow[n\to\infty]{\P}}

\newcommand{\dconv}{\xRightarrow{n\to\infty}}

\allowdisplaybreaks

\begin{document}

\title[CLT for the average target hitting time in Erd\H{o}s-Renyi graphs]
{A Central Limit Theorem for the average target hitting time
for a random walk on a random graph}

\author[Matthias L\"owe]
{Matthias L\"owe }
\address[Matthias L\"owe]{Fachbereich Mathematik und Informatik,
Universit\"at M\"unster,
Einsteinstra\ss e 62,
48149 M\"unster,
Germany}

\email[Matthias L\"owe]{maloewe@uni-muenster.de}

\author[Sara Terveer]
{Sara Terveer}
\address[Sara Terveer]{Fakult\"at f\"ur Mathematik,
Universit\"at Bielefeld,
Postfach 100131,
33501 Bielefeld,
Germany}

\email[Sara Terveer]{sterveer@math.uni-bielefeld.de}

\subjclass[2010]{}

\thanks{Research was
funded by the Deutsche Forschungsgemeinschaft (DFG, German Research Foundation) under Germany 's Excellence Strategy
EXC 2044-390685587, Mathematics M\"unster: Dynamics-Geometry-Structure}

\maketitle

\begin{abstract}
Consider a simple random walk on a realization of an Erd\H{o}s-R\'enyi graph. Assuming that it is asymptotically almost surely (a.a.s.)\ connected, we prove a Central Limit Theorem (CLT) for the average target hitting time. The latter is the expected time it takes the random walk on average to first hit a fixed vertex $j$. The average is taken over all possible starting vertices, with respect to $\pi$, the invariant measure of the random walk.
\end{abstract}

\section{Introduction}\label{sec:intro}
Let $G_n=(V_n,E_n)$ be a realization of an \ERG{} $\mathcal{G}(n,p)$, i.e.\ we set $V_n:=\{1, \ldots ,n\}$ and we connect any two 
vertices $i \neq j \in V_n$ by an undirected edge $\{i,j\}$ with probability $p$, independently of all other edges. We collect the 
resulting edges in $E_n$. We will choose the probability $p=p_{n}$ 
such that it may and typically will depend on $n$
and such that $n p_n$ diverges to infinity. 
More precisely, we will require that for some fixed $\xi>1$,
\begin{equation}
	\frac{(\log n)^{6\xi}}{np_n}\to 0
	\label{eq:connectivity}
\end{equation}
as $n\to\infty$ to ensure that the random graph is asymptotically almost surely connected and our results always tacitly assume 
connectivity of the graph. The parameter $\xi$ is immaterial for most of the 
proof and arises solely as an artifact of a spectral property of (a close relative of) the adjacency matrix of $G_n$ 
to be discussed later (see Lemma \ref{lem:spectralgap} below).

We will additionally require
\begin{equation}
p\leq \overline{p}< 1
\label{eq:pbounded}
\end{equation} for some fixed $\overline{p}\in(0,1)$.
%
In the following we omit the index $n$ whenever suitable. Unless otherwise remarked, the calculations are made on $G=G_n=(V_n,E_n)=(V,E)$.

For fixed $n$ consider the simple random walk in discrete time $(X_t)$ on $G$: If $X_t$ is in the vertex $i\in V$ 
at time $t$, $X_{t+1}$ will be in the vertex
$j$ with probability $\frac{1}{d_i}$, where $d_i$ denotes the degree of $i$, if $\{i,j\}\in E$, and with probability 0, otherwise.
The invariant distribution of this walk is given by
\[\pi_i\coloneqq \frac{d_i}{\sum\limits_{j\in V}d_j}=\frac{d_i}{2|E|}.\]
Let $H_{ij}$ be the expected time it takes the walk to reach vertex $j\in V$ when starting from vertex $i\in V$. Of course, $H_{ij}$
will generally, among others, depend on the graph distance of $i$ and $j$. To compensate for this one introduces  
\begin{equation}\label{eq:defhitting}
H_j:=\sum\limits_{i\in V}\pi_iH_{ij}\quad\text{ and }\quad H^i:=\sum\limits_{j\in V}\pi_j H_{ij}
\end{equation}
These quantities are called the \emph{mean target hitting time} and the \emph{mean starting hitting time}, respectively. 
Note that $H_j$ and $H^i$ are expectation values in the random walk measure, but with respect to the realization of the random graph, 
they are random variables.
In \cite{LT14}, the asymptotic behaviour of $H_j$ and $H^i$ was analyzed on the level of a Law of Large Numbers. It was shown that
\[H_j=n(1+o(1))\quad\text{ as well as }\quad H^i=n(1+o(1)) \]
asymptotically almost surely proving a conjecture from \cite{sood}. In \cite{HL19} this result was extended to random hypergraphs. Such results can be considered Laws of Large Numbers for the average target or starting hitting times, respectively. After having
proved them, a natural question 
is the one about fluctuations around such a Law of Large Numbers. 
In  \cite{loewe2020} we proved a Central Limit Theorem (CLT) for the mean starting hitting time $H^i$
 based on a CLT for incomplete U-statistics over a triangular array of random variables proved in \cite{loewe2020A}. However, this technique cannot be used to also 
prove a CLT for the mean target hitting time.

The goal of this note is to also prove a CLT for $H_j$. Our main result will be
\begin{thm}\label{thm:main}
	Let $H_j$ be defined as in \eqref{eq:defhitting} and assume $p$ satisfies \eqref{eq:connectivity} and \eqref{eq:pbounded}. Then
	\[\sqrt{\frac{p}{n(1-p)}}\bigl(H_j-n\bigr)\dconv \mathcal{N}(0,1),\]
	where $\Rightarrow$ denotes convergence in distribution and $\mathcal{N}(0,1)$ is a standard normally distributed random variable.
\end{thm}

To prove this theorem, we will proceed as follows: In Section \ref{sec:spectraldecomposition}, we will start by recalling a classic spectral decomposition of the hitting times by \cite{Lovasz93}. The resulting term (see \eqref{repr_target} below) will be considered in two parts: The first one can be treated using an application of (a refined version of) the Delta method in Section \ref{sec:delta-method}. It will turn out that this term is responsible for the Gaussian fluctuations. The second part 
is a sum over inverse eigenvalues multiplied with components of the eigenvectors of a transformed adjacency matrix of the realization of the random graph. We will show that it is negligible on the scale of the CLT in Theorem \ref{thm:main}. Finally, in Section \ref{final}, we prove the main theorem by merging the aforementioned results and, yet again, applying the Delta method. Some auxiliary results are given in the Appendix.

Throughout this note, we will use the notation $a_n\approx b_n$ for $a_n=b_n(1+o(1))$ and $a_n \lesssim b_n$, if $a_n \le b_n(1+o(1))$.

\section{Spectral decomposition of the hitting times}\label{sec:spectraldecomposition}
We start as in \cite{LT14}, the spectral decomposition of the hitting times, taken from \cite{Lovasz93}, Section 3.
To introduce it, let $A$ be the adjacency matrix of $G$, i.e.\ $A=(a_{ij})$, with $a_{ij}=\1
_{\{i,j\}\in E}$ as above.
Let $D$ be the diagonal matrix $D=(\mathrm{diag}(d_i))_{i=1}^{n}$.
With $A$ we associate the matrix $B$ defined as $B:=D^{-\frac 12}AD^{-\frac 12}$. The matrix $B$ is intrinsically related to the symmetrically normalized Laplacian matrix $L$ of $G$ defined as $L:=\operatorname{Id}-B$.
Note that $B=(b_{ij})$ with $$b_{ij}=\frac{a_{ij}}{\sqrt{d_i d_j}}.$$ Therefore, $B$ is symmetric and hence has real eigenvalues.
$\lambda_1 = 1$ is an eigenvalue of $B$, since $w:=(\sqrt{d_1}, \cdots, \sqrt{d_{n}})$ satisfies $Bw=w$ and
by the Perron-Frobenius theorem $\lambda_1$ is the largest eigenvalue.
We order the eigenvalues $\lambda_k$ of $B$ such that $$\lambda_1 \geq \lambda_2 \geq \cdots \geq \lambda_{n} $$
and we normalize the eigenvectors $v_k$ to the eigenvalues $\lambda_k$ to have length one. Thus, in particular,
$$	v_1 := \tfrac{w}{\sqrt{2|{E}|}}= \left(\sqrt{\tfrac{d_j}{2|{E}|}} \right)_{j=1}^{n}= \left(\sqrt{\pi_j} \right)_{j=1}^{n}.$$
Also recall that the matrix of the eigenvectors is orthogonal and the scalar product of two eigenvectors $v_i$ and $v_j$  satisfies $\langle v_i, v_j\rangle =\delta_{ij}$.
With this, one has the following representation of $H_j$: 
\begin{prop}[cf. {\cite{Lovasz93}, Theorem 3.1 and Formula (3.3)}]
	For all $i\neq j \in V$ we have
	$
	H_{ij}=2 |{E}| \sum_{k=2}^{n} \frac{1}{1- \lambda_k} \left(   \frac{v_{k, j}^2}{d_j}   -   \frac{v_{k, i}  v_{k, j}}{\sqrt{d_i d_j}}     \right).
	$
	Thus,
	\begin{equation}\label{repr_target}
	H_j=\frac{2|E|}{d_j}\sum_{k=2}^{n} \frac{1}{1-\lambda_k}v_{kj}^2.
	\end{equation}
\end{prop}

Notice that the terms in the sum cannot be easily controlled because a priori we know little about the eigenvectors of $B$. This is
in contrast to the situation where one considers the eigenvectors of $A$. For them, in \cite{Erdoes13} the authors showed complete
delocalization (with the exception of the Perron-Frobenius eigenvector, of course). 

We will subsequently decompose the expression in \eqref{repr_target} into the term in front of the sum and the sum itself and analyze both separately.

\section{The Delta Method}\label{sec:delta-method}

We begin with the factor in front of the sum in \eqref{repr_target}: If we denote by $N_j$ those edges that link vertex $j$ to another vertex, we have that $|N_j|=d_j$ and thus
\[\frac{2|E|}{d_j}=\frac{2|E\setminus N_j|}{d_j}+2,\]
where $E\setminus N_j$
denotes the set of edges excluding those that link $j$ to other vertices.
In a first step, we will prove that $\log \frac{|E\setminus N_j|}{d_j}$ is asymptotically normally distributed.

\begin{prop} 
	\label{eq:distconv}Under the assumptions of Theorem \ref{thm:main},
	\begin{equation}
		\sqrt{\frac{np}{1-p}}\left(\log\left(2 \frac{|E\setminus N_j|}{d_j}+2\right)-\log n\right)\dconv\mathcal{N}\left(0,1\right).
	\end{equation}
\end{prop}
We will prepare the proof of Proposition \ref{eq:distconv} by a lemma. 
To this end, let us fix a vertex $j$ and have a closer look at $E\setminus N_j$.
Notice that 
\[|E\setminus N_j|=\sum_{\substack{i<k\\ i,k\neq j}}a_{ik}.\]
We can immediately prove a result on convergence in distribution for this last term:
\begin{lem}\label{lem:lindebergfeller}
	By the Lindeberg-Feller CLT for triangular arrays of random variables,
	 we have
	\[\sqrt{\frac{\binom{n-1}{2}}{p(1-p)}}\Biggl(\frac{|E\setminus N_j|}{\binom{n-1}{2}}-p\Biggr)\dconv \mathcal{N}(0,1).\]
	\end{lem}
\begin{proof}
	For $i=2,\dots,n$, $i\neq j$ let $S_{n,i}\coloneqq \sum\limits_{\substack{k=1\\k\neq j}}^{i-1}a_{ik}$ and set 
	$$Y_{n,i}\coloneqq\frac{1}{\sqrt{\binom{n-1}{2}p(1-p)}}(S_{n,i}-\E S_{n,i}).$$ 
	Obviously, the $S_{n,i}$ and therefore the $Y_{n,i}$ are independent in $i$ for fixed $n$ by construction. Furthermore, $\E Y_{n,i}=0$.
Denoting by $X \sim \mathrm{Bin}(N,q)$ the binomial distribution with parameters $N$ and $q$ of a random variable, from $S_{n,i}\sim\mathrm{Bin}(i-1,p)$ for $i< j$ and $S_{n,i}\sim\mathrm{Bin}(i-2,p)$ for $i>j$ we obtain: 
	\begin{align*}\sum\limits_{\substack{i=1\\i\neq j}}^n\E Y_{n,i}^2&=\sum\limits_{\substack{i=1\\i\neq j}}^n\frac{\mathbb{V} S_{n,i}}{\binom{n-1}{2}p(1-p)}=\frac{1}{\binom{n-1}{2}p(1-p)}\left(\sum\limits_{i=1}^{j-1}(i-1)p(1-p)+\sum\limits_{i=j+1}^n(i-2)p(1-p)\right)\\
		&=\frac{1}{\binom{n-1}{2}}\left(\sum\limits_{i=1}^n(i-1)-(j-1)-\sum\limits_{i=j+1}^n1\right)=\frac{1}{\binom{n-1}{2}}\left(\binom{n}{2}-(n-1)\right)=1.\\
	\end{align*}
	Furthermore,  using that the fourth central moment of a $B(N,q)$ distributed random variable $Z$ is given by
$
\E\left[\left(Z-\E Z\right)^4\right]= Nq(1-q)\left(1+3(N-2)q(1-q)\right)
$	
	we find that
	\begin{align*}E Y_{n,i}^4&=\frac{1}{\binom{n-1}{2}^2p^2(1-p)^2}\EW{(S_{n,i}-\E S_{n,i})^4}\\
	&\le \frac{1}{\binom{n-1}{2}^2p^2(1-p)^2}ip(1-p)\left(1+3(i-2)p(1-p)\right)\\
		&\lesssim \frac{np(1-p)3(n-2)p(1-p)}{\binom{n-1}{2}^2p^2(1-p)^2}=\frac{3n(n-2)}{\binom{n-1}{2}^2}\approx\frac{12}{n^2}
	\end{align*}
Thus, by the Cauchy-Schwarz inequality and Markov's inequality we estimate that
\[\sum\limits_{\substack{i=1\\i\neq j}}^n\EW{Y_{n,i}^2\ind{{|Y_{n,i}|>\delta}}}\leq \sum\limits_{\substack{i=1\\i\neq j}1}^n\sqrt{\EW{Y_{n,i}^4}\WK{|Y_{n,i}|>\delta}}\leq \sum\limits_{\substack{i=1\\i\neq j}}^n\sqrt{\frac{\EW{Y_{n,i}^4}^2}{\delta^4}}\lesssim n \frac{12}{n^2}\frac{1}{\delta^2}\nconv 0.\] 
Therefore, the Lindeberg-condition holds and we can apply the Lindeberg-Feller CLT for triangular arrays to $((Y_{n,i})_i)_n$. 
This immediately yields
$$\sum\limits_{i=1}^nY_{ni}\dconv \mathcal{N}(0,1).$$
From $\sum\limits_{\substack{i=1\\i\neq j}}^nS_{n,i}=|E\setminus N_j|$ and thus $\sum\limits_{\substack{i=1\\i\neq j}}^n\E{S_{n,i}}=\binom{n-1}{2}p$ we arrive at
\[\tfrac{1}{\sqrt{\binom{n-1}{2}p(1-p)}}\Big(|E\setminus N_j|-\tbinom{n-1}{2}p\Big)=\sum\limits_{\substack{i=1\\i\neq j}}Y_{n,i}\dconv \mathcal{N}(0,1).\]
This is our claim.
\end{proof}

We would like to apply this lemma by applying the Delta method in the form of Theorem \ref{delta-method} in the appendix. To do so, we need to define uniform continuity of a function over a sequence:
\begin{defi}\label{defn:unif-cont-seq}
	Consider an open set $D\subset\R^k$, $k\in\N$ and a sequence of vectors $(x_n)_{n \in \N} \subset D$. 
	For $l\in\N$, let $f:D\to\R^l$ be a continuous function (on $D$). We call $f$ 
	\emph{uniformly continuous with respect to the sequence $(x_n)$}\index{continuous, uniformly with respect to a sequence} 
	if for any $\eps>0$ there is a $\delta>0$ and an $N\in\N$, so that for all $z\in\R^k$ with $\|z\|<\delta$ and $n\geq\N$
	\[\|f(x_n+z)-f(x_n)\|<\eps.\]  
	Here, by $\|\cdot\|$, we denote the Euclidean norm.                  
\end{defi}

With this definition and the Delta method as stated in the appendix in mind, we can prove the proposition:

\begin{proof}[Proof of Proposition \ref{eq:distconv}]
Note that by Levy's CLT, we also know that
\[\tfrac{1}{\sqrt{(n-1)p(1-p)}}(d_j-(n-1)p)\dconv \mathcal{N}(0,1).\]
Thus by independence of $d_j$ and $E\setminus N_j$ and Lemma \ref{lem:lindebergfeller}:
\[\left(\frac{|E\setminus N_j|}{\sqrt{\binom{n-1}{2}p(1-p)}},\frac{d_j}{\sqrt{(n-1)p(1-p)}}\right)-\left(\frac{\binom{n-1}{2}p}{\sqrt{\binom{n-1}{2}p(1-p)}},\frac{(n-1)p}{\sqrt{(n-1)p(1-p)}}\right)\dconv \mathcal{N}(0,I),\]
where $I$ is the $2 \times 2$ identity matrix.
Consequently
\[r_n\cdot(T_n-\theta_n)\dconv \mathcal{N}(0,I)=T\]
with
\begin{equation} \label{eq:defsVariablen}
	r_n:=\sqrt{\tfrac{(n-1)p}{1-p}},\text{ }T_n:=\left(\tfrac{|E\setminus N_j|}{\sqrt{\binom{n-1}{2}(n-1)p^2}},\tfrac{d_j}{(n-1)p}\right)^t,\text{ and }\theta_n=\left(\sqrt{\tfrac{n-2}{2}},1\right)^t
\end{equation}

Now consider the situation of Definition \ref{defn:unif-cont-seq} with $k=2$ and $D=(0,\infty)^2$.
Then $\theta_n\in D$ for $n\geq 3$. 
Moreover, consider the continuous differentiable function 
\begin{align*}
\phi:D&\to\R ,\quad
(x,y)\mapsto \log \tfrac{x}{y}
\end{align*} 
Note that $\nabla \phi(x,y)=
\left(\frac{1}{x},-\frac{1}{y}\right)$ is uniformly continuous on $D'=(\frac{1}{2},\infty)^2$, and since $\theta_n\in D'$ for all $n\geq 3$, 
$\nabla\phi(x,y)$ is uniformly continuous 
with respect to $\theta_n$.
Furthermore, $\lim\limits_{n\to\infty}\nabla\phi(\theta_n)=(0,-1)^t$. 
Thus, we can apply Theorem \ref{delta-method}
\[r_n(\phi(T_n)-\phi(\theta_n))\dconv \left(\lim\limits_{n\to\infty}\nabla\phi(\theta_n)\right)^t\cdot T.\]
Since $T$ is a two-dimensional normal distribution with expectation vector $0$ and identity covariance matrix,
multiplication by $(0,-1)^t$ gives a one-dimensional standard normal distribution $\mathcal{N}(0,1)$.
Hence
\[\sqrt{\tfrac{(n-1)p}{1-p}}\left[\log \left(\sqrt{\tfrac{2}{n-2}}\cdot \tfrac{|E\setminus N_j|}{d_j}\right)-\log\left(\sqrt{\tfrac{n-2}{2}}\right)\right]\dconv \mathcal{N}\left(0,1\right)\]
which leads to
\[\sqrt{\tfrac{(n-1)p}{1-p}}\left(\log\left(2 \tfrac{|E\setminus N_j|}{d_j}\right)-\log(n-2)\right)\dconv\mathcal{N}\left(0,1\right).\]
For another application of the Delta method, let $k=1$,
$$
D=(0,\infty),\, r_n=\sqrt{\tfrac{(n-1)p}{1-p}},\,T_n=\log\left(2 \tfrac{|E\setminus N_j|}{d_j}\right), \,\theta_n=\log(n-2), \,
\mbox{and }\phi(x)=\log(e^x+2).
$$ 
Note that again $\phi$ is continuously differentiable on $D$ with $\nabla\phi(x)=\frac{1}{1+\frac{2}{\exp(x)}}$ and $\lim\limits_{n\to\infty}\nabla\phi(\theta_n)=1$. 
Furthermore, $\nabla\phi$ is uniformly continuous with respect to $\theta_n$.
By Theorem \ref{delta-method}, we obtain 
$\sqrt{\frac{(n-1)p}{1-p}}\left(\log\left(2 \frac{|E\setminus N_j|}{d_j}+2\right)-\log n\right)\dconv\mathcal{N}\left(0,1\right)$  
which was our claim (equivalently, one could apply Slutzky's theorem).
\end{proof}

\section{Negligibility of the remaining term}
The remaining term is a composition of non-Perron eigenvalues of $B$ and entries of their corresponding eigenvectors. 
We begin by stating the following result on the spectral gap:
\begin{lem}\label{lem:spectralgap}
	Under the assumptions of Theorem \ref{thm:main}, with probability converging to 1 we have 
	for all $k=2,\dots, n$
	\[|\lambda_k|\lesssim \frac{2}{\sqrt{np}}.\]
\end{lem}

\begin{proof}
	Let $X$ have a $\mathrm{Bin}(n-2,p)$-distribution and define
\[\mu:=\EW{\tfrac{1}{X+1}}\quad\text{and}\quad\tau:=\EW{\tfrac{1}{\sqrt{{X+1}}}}.\]
Recall that by a result from \cite{Znidaric}
for a random variable $X$ with $X\sim\mathrm{Bin}(n,p)$ we have
	\[\EW{\left(\tfrac{1}{X+1}\right)^r}\approx\tfrac{1}{(np)^r}\]
	for all $r>0$.	
	Now let $$B'=\frac{1}{\rho}B$$ 
	with  
	$$\rho^2=(n-1)(p\mu^2-p^2\tau^4)\leq np\mu^2\approx \tfrac{1}{np}.
	$$ 
	Denote by $\mu_1\geq \dots\geq \mu_{n}$, $k=1,\dots,n$ the eigenvalues of $B'$. Since $B'$ satisfies Definition 2.2 in \cite{Erdoes13}, by Corollary 6.6 in the same paper,
	\begin{equation}
		|\mu_k|\lesssim 2+o(1)
		\label{eq:eigenvalueGapB'}
	\end{equation}
	for all $k=2,\dots,n$ with probability converging to 1
	(note, however, the different indexing: we sorted eigenvalues in descending order, \cite{Erdoes13} sort them increasingly. Furthermore, notice the difference in nomenclature: in \cite{Erdoes13} $\lambda_k$, $k=1,\dots,n$ denotes the eigenvalues of $H$, unlike here, where they denote the eigenvalues of $B$). We obtain that $\lambda_k=\rho\mu_k$ for the eigenvalues $\lambda_k$, $k=1,\dots,n$ of $B$.
	 Thus, for $k=2,\dots,n$ by \eqref{eq:eigenvalueGapB'}
	\[\lambda_k=\rho\mu_k\lesssim \tfrac{2}{\sqrt{np}}\]
	with probability converging to 1 and the claim holds.
\end{proof}

We will now prove that the sum in \eqref{repr_target} is asymptotically negligible on the appropriate scale:
\begin{prop}\label{eq:secondpartnegligible}
	Under the assumptions of Theorem \ref{thm:main}, 
	\begin{equation*}
		\sqrt{\frac{np}{1-p}}\log\left(\sum\limits_{k=2}^n\frac{1}{1-\lambda_k}v_{kj}^2\right)\Pconv 0,
	\end{equation*}
	as $n\to\infty$ for every $j\in\{1,\dots,n\}$.
\end{prop}
\begin{proof}
To prove the proposition, let us rewrite the term in the logarithm.
By expanding the geometric series (note that $|\lambda_k|<1$),
\begin{align*}Z_n\coloneqq & \sum\limits_{k=2}^{n}\frac{1}{1-\lambda_k}v_{kj}^2=\sum\limits_{k=2}^{n}\sum\limits_{m=0}^\infty\lambda_k^mv_{kj}^2=\sum\limits_{k=2}^{n}\Big(1+\lambda_k+\lambda_k^2\sum\limits_{m=0}^\infty\lambda_k^m\Big)v_{kj}^2
= \sum\limits_{k=2}^{n}\Big(1+\lambda_k+\frac{\lambda_k^2}{1-\lambda_k}\Big)v_{kj}^2
\end{align*}
As seen earlier in Section \ref{sec:spectraldecomposition} we have that $v_{1j}^2=\pi_j$, $\sum_{k=1}^{n}v_{kj}^2=1$ and 
moreover,
\[\sum\limits_{k=1}^{n}\lambda_kv_{kj}^2=\sum\limits_{k=1}^{n}\sum\limits_{i=1}^nb_{ji}v_{ki}v_{kj}=\sum\limits_{i=1}^nb_{ji}\cdot \langle v_i,v_j\rangle=b_{jj}=0.\] 
 Thus we arrive at
\begin{equation*}
	Z_n=\sum\limits_{k=2}^{n}v_{kj}^2+\sum\limits_{k=2}^{n}\lambda_kv_{kj}^2+\sum\limits_{k=2}^{n}\frac{1}{1-\lambda_k}\lambda_k^2v_{kj}^2=1-2\pi_j+\sum\limits_{k=2}^{n}\frac{\lambda_k^2}{1-\lambda_k}v_{kj}^2.
\end{equation*}
Finally, making use of Lemma \ref{lem:spectralgap}
\begin{equation*}
	0\leq \sum\limits_{k=2}^{n}\frac{\lambda_k^2}{1-\lambda_k}v_{kj}^2 = O\Big(\tfrac{1}{np}\Big)\cdot \sum\limits_{k=2}^{n}\frac{1}{1-\lambda_k}v_{kj}^2= O\Big(\tfrac{1}{np}\Big)\cdot Z_n, 
\end{equation*}
with probability converging to 1 and hence	
$$1-2\pi_j \leq Z_n \quad \mbox{ as well as }Z_n \leq 1-2\pi_j+Z_n\cdot O\Big(\tfrac{1}{np}\Big)$$  
which is equivalent to saying that 
\begin{align}
-2\pi_j\leq Z_n-1\leq -2\pi_j\cdot(1+o(1))+O\Big(\tfrac{1}{np}\Big)
	\label{eq:secondterm}
\end{align}
with probability converging to 1.

As $\pi_j=\frac{d_j}{2|E|}$ for all $j=1, \ldots, n$ 
it is immediate to see that the $\pi_j$ are identically distributed for $j=1,\dots,n$. 
Furthermore,
$\sum_{j=1}^{n}d_j=2|E|$.
Thus obviously, we have
\[n\EW{\pi_j}=\sum\limits_{j=1}^{n}\EW{\pi_j}=\frac{1}{2|E|}\E\Bigl[\sum\limits_{j=1}^{n}d_j\Bigr]=1.\]
Therefore,
\[\E\Bigl[\sqrt{\tfrac{np}{1-p}}\pi_j\Bigr]=\sqrt{\tfrac{p}{n(1-p)}}\nconv 0\]
and by non-negativity of $\sqrt{\frac{np}{1-p}}\pi_j$ and Markov's inequality,
\begin{equation}
	\sqrt{\tfrac{np}{1-p}}\pi_j\Pconv 0,
	\label{eq:pi-term}
\end{equation}
and therefore together with \eqref{eq:secondterm}, $\sqrt{\frac{np}{1-p}}(Z_n-1)\Pconv 0$, using that $np(1-p)\to\infty$ as $n\to\infty$. By Lemma \ref{lem:convproblog}, we therefore obtain 
\[\sqrt{\tfrac{np}{1-p}}\log Z_n\Pconv 0,\]
proving the claim.
\end{proof}

\section{Proof of Theorem \ref{thm:main}}\label{final}
We proceed by completing the proof, using the decomposition \eqref{repr_target}
\begin{align*}
	\sqrt{\frac{np}{1-p}}\bigl(\log (H_j)-\log n\bigr)\hspace{-2cm}&\hspace{2cm}=\sqrt{\frac{np}{1-p}}\left[\log\left(2 \frac{|E\setminus N_j|}{d_j}+2\right)+\log\left(\sum\limits_{k=2}^{n}\frac{1}{1-\lambda_k}v_{kj}^2\right)-\log n\right]\\
	&=\sqrt{\frac{np}{1-p}}\left[\log\left(2 \frac{|E\setminus N_j|}{d_j}+2\right)-\log n\right]+\sqrt{\frac{np}{1-p}}\log\left(\sum\limits_{k=2}^{n}\frac{1}{1-\lambda_k}v_{kj}^2\right).
\end{align*}
The first term converges to a standard Gaussian random variable in distribution by Proposition \ref{eq:distconv}. The second part converges to 0 in probability by Proposition \ref{eq:secondpartnegligible}. By Slutzky's theorem we obtain
\begin{equation}
\sqrt{\tfrac{np}{1-p}}\left(\log {H_j}-\log{n}\right)\dconv \mathcal{N}(0,1)
\label{eq:logCLT}
\end{equation}
We apply one more instance of the Delta method: Choose $k=2$, 
$$
D=(\R^+)^2,\, r_n=\sqrt{\tfrac{np}{1-p}},\, T_n=(\log {H_j},n)^t,\, \theta_n=(\log n,n)^t, \mbox{  and }
\phi(x,y)=\tfrac{1}{y}\exp(x).$$ 
Notice that $\phi$ is continuously differentiable with 
$$\nabla\phi(x,y)=\left(\tfrac{1}{y}\exp(x),-\tfrac{1}{y^2}\exp(x)\right)^t,$$ 
and that $\lim\limits_{n\to\infty}\nabla\phi(\theta_n)=(1,0)^t$.

Furthermore, $\nabla\phi$ is uniformly continuous with respect to $\theta_n$: Let $\eps>0$ arbitrary, but fixed, and choose $\delta=\log(\frac\eps2+1)$ and $z=(z_1,z_2)$ with $0<\|z\|<\delta$. Without loss of generality, let $z_1,z_2>0$ (for negative $z_1$ or $z_2$, the procedure is analogous).
\begin{align*}
	\nabla\phi(\theta_n\!+\!z)-\nabla\phi(\theta_n)&=\Big(\tfrac{1}{n\!+\!z_2}\exp(\log n\!+\!z_1),-\tfrac{1}{(n\!+\!z_2)^2}\exp(\log(n\!+\!z_1))\Big)^t-\Big(1,-\tfrac{1}{n}\Big)^t\\
	&=\left(\tfrac{1}{n+z_2}ne^{z_1}-1,-\tfrac{1}{(n+z_2)^2}ne^{z_1}-\tfrac{1}{n}\right)^t
\end{align*}
The second component is of order $O\left(\frac{1}{n}\right)$. Thus, 
\begin{align*}
	\|\nabla\phi(\theta_n\!+\!z)-\nabla\phi(\theta_n)\|&=\sqrt{\left(\tfrac{1}{n+z_2}ne^{z_1}-1\right)^2+O\left(\tfrac{1}{n}\right)^2}\\
	&\leq\left|\tfrac{n(e^{z_1}-1)-z_2}{n+z_2}\right|+O\left(\tfrac{1}{n}\right)\leq\left|e^{z_1}-1\right|+\tfrac{z_2}{n}+O\left(\tfrac{1}{n}\right)=e^{z_1}-1+O\left(\tfrac{1}{n}\right),
\end{align*}
since $z_1>0$. Now by $z_1<\|z\|<\delta$, $e^{z_1}<e^\delta=\frac\eps2+1$. Additionally, there is some $N\in\N$ so that for all $n\geq N$, the $O\left(\frac{1}{n}\right)$-term is bounded by $\frac\eps2$. Thus, for all $n\geq N$ and all $z\in D$ with $\|z\|<\delta$,
\[\|\nabla\phi(\theta_n\!+\!z)-\nabla\phi(\theta_n)\|<\eps,\]
which proves uniform continuity with respect to $\theta_n$. By Theorem \ref{delta-method}:
\begin{equation*}
	\sqrt{\tfrac{np}{1-p}}\left(\tfrac{H_j}{n}-1\right)=\sqrt{\tfrac{np}{1-p}}\tfrac{H_j-n}{n}=\sqrt{\tfrac{p}{n(1-p)}}\bigl(H_j-n\bigr)\dconv \mathcal{N}(0,1).
	\label{eq:CLTcomplete}
\end{equation*}
This completes the proof of \cref{thm:main}.\qed
\appendix
\section{Auxiliary Results}\label{sec:appendix}
The following lemma shows the convergence in probability of logarithmic terms:
\begin{lem}\label{lem:convproblog}
	Let $a_n$ be a sequence of non-negative numbers and $X_n$ a sequence of random variables, so that, as $n\to\infty$, $a_n\to \infty$ and $a_nX_n\to 0$ in probability. Then $a_n\log(1+X_n)$ converges to 0 in probability.
\end{lem}

\begin{proof}
	Let $\eps>0$ and choose $\zeta>0$ such that $\zeta+\frac{\zeta^2}{2}=\eps$. Since $a_nX_n\to 0$ in probability, $\P(|a_nX_n|\leq \zeta)\to 1$, so for every $\delta>0$ there is a $N(\delta)\in\N$ such that for all $n\geq N(\delta)$, $\P(|a_nX_n|\leq \zeta)\geq 1-\delta$ and $a_n\geq 1$ (since $a_n\to\infty$). Using a Taylor series, there is an $\alpha\geq 0$ such that $\log(1+x)=x-\frac{1}{2(1+\alpha)}x^2$, therefore for $\omega\in\Omega_{n,\zeta}\coloneqq\{|a_nX_n(\omega)|\leq \zeta\}$, 
	\[|a_n\log(1+X_n(\omega))|=\Big|a_nX_n(\omega)-\tfrac{a_n}{2(1+\alpha)}X_n^2(\omega)\Big|\leq |a_nX_n(\omega)|+\Big|\tfrac{a_n}{2}X_n^2(\omega)\Big|\leq \zeta+\tfrac{\zeta^2}{2}=\eps.\]
	Thus, for every $\delta>0$, there is a $N(\delta)\in\N$ such that for all $n\geq N(\delta)$,
	\[\P(|a_n\log(1+X_n(\omega))|\leq\eps)\geq \P(|a_nX_n|\leq \zeta)\geq 1-\delta,\]
	i.e. $\P(|a_n\log(1+X_n(\omega))|\leq\eps)\to1$. Since $\eps>0$ was chosen arbitrarily, $a_n\log(1+X_n(\omega))\to 0$ in probability.
\end{proof}

The uniform Delta method, c.f. Theorem 3.8. in \cite{Vaart98}, is a crucial point in our analysis. We extend the result by van der Vaart to include the setting where one or multiple components of $\theta_n$ tend to infinity as $n$ increases: 

\begin{thm}[Delta-Method]\label{delta-method}
	Consider an open set $D\subset\R^k$ and random vectors $T_n$ with values in $D$ and satisfying $r_n(T_n-\theta_n)\dconv T$  for vectors $\theta_n\in D$, scalars $r_n\to\infty$ and a random vector $T\in\R^k$. Consider a  continuously differentiable function $\phi:D\to\R$. Furthermore, assume that $\nabla\phi$ is uniformly continuous with respect to $\theta_n$ and that $\lim\limits_{n\to\infty}\nabla\phi(\theta_n)$ exists. Then 
	\[r_n\left(\phi(T_n)-\phi(\theta_n)\right)\dconv \bigl(\lim\limits_{n\to\infty}\nabla\phi(\theta_n)\bigr)^t\cdot T.\] 
	By $\nabla\phi$ we denote the gradient of $\phi$.
\end{thm}
\begin{proof}
	For fixed $h\in \R^k$ and $s\in[0,1]$, define $g_n(s)=\phi(\theta_n+s h)$. Because $\phi$ is continuously differentiable, $g_n:[0,1]\to\R$ is also continuously differentiable with derivative $g_n'(s)=\bigl(\nabla\phi(\theta_n+s h)\bigr)^t\cdot h$. By the 
	mean-value theorem, $g_n(1)-g_n(0)=g_n'(\xi)$ for some $\xi\in[0,1]$. Thus,
	\begin{align}
		R_n(h)&:=\phi(\theta_n+h)-\phi(\theta_n)-\bigl(\lim\limits_{n\to\infty}\nabla\phi(\theta_n)\bigr)^t\cdot h\label{eq:R_ndef}\\
		&=g_n(1)-g_n(0)-\bigl(\lim\limits_{n\to\infty}\nabla\phi(\theta_n)\bigr)^t\cdot h\notag\\
		&=g_n'(\xi)-\bigl(\lim\limits_{n\to\infty}\nabla\phi(\theta_n)\bigr)^t\cdot h\notag\\
		&=\bigl(\nabla\phi(\theta_n+\xi h)\bigr)^t\cdot h-\bigl(\lim\limits_{n\to\infty}\nabla\phi(\theta_n)\bigr)^t\cdot h\notag\\
		&=\Big[\bigl(\nabla\phi(\theta_n+\xi h)\bigr)-\bigl(\lim\limits_{n\to\infty}\nabla\phi(\theta_n)\bigr)\Big]^t\cdot h.
		\label{eq:R_ndef2}
	\end{align}
	Now denote by $\|\cdot\|$ the Euclidean norm. By the triangle inequality,
	\begin{align*}\Big\|\bigl(\nabla\phi(\theta_n+\xi h)\bigr)^t-\bigl(\lim\limits_{n\to\infty}\nabla\phi(\theta_n)\bigr)^t\Big\|\hspace{-3cm}&\\
		&\leq\left\|\left(\nabla\phi(\theta_n+\xi h)\right)^t-\left(\nabla\phi(\theta_n)\right)^t\right\|+\Big\|\left(\nabla\phi(\theta_n)\right)^t-\bigl(\lim\limits_{n\to\infty}\nabla\phi(\theta_n)\bigr)^t\Big\|.
	\end{align*}
	For any $\eps>0$, the second term on the right hand side is bounded by $\eps/2$ for sufficiently large $n$ by definition of the limit.
	Furthermore, by uniform continuity of $\nabla\phi$ with respect to $\theta_n$, there is some $\delta>0$ so that the first term 
	on the right hand side is also bounded by $\eps/2$ as long as $\left\|\xi h\right\|<\delta$. Thus, if $\|h\|<\delta$ and therefore 
	by $0\leq \xi\leq 1$ also $\|\xi h\|<\delta$, we find by \cref{eq:R_ndef2}
	\begin{align}
		\begin{split}
			|R_n(h)|&=\big|\big\langle\nabla\phi(\theta_n+\xi h)-\lim\limits_{n\to\infty}\nabla\phi(\theta_n), h\big\rangle\big|\\
			&\leq\big\|\nabla\phi(\theta_n+\xi h)-\lim\limits_{n\to\infty}\nabla\phi(\theta_n)\big\|\cdot \|h\|<\eps\|h\| <\eps\delta
			\label{eq:R_nabsch}\end{split}
	\end{align} 
	for sufficiently large $n$.
	For any $\eta>0$, denote  $A_n=\big\{\|R_n(T_n-\theta_n)\|\geq\frac{\eta}{r_n}\big\}$ and $B_n=\big\{\|T_n-\theta_n\|\geq \delta\big\}$. Then, by \cref{eq:R_nabsch}, $A_n\cap B_n^c\subseteq \left\{\eps\delta\geq \frac{\eta}{r_n}\right\}$.
	Since $r_n\to\infty$, this set is asymptotically empty, therefore $\WK{A_n\cap B_n^c}\nconv 0$.
	On the other hand, $$\WK{A_n\cap B_n}\leq \WK{B_n}\nconv 0,$$ because from $r_n(T_n-\theta_n)\dconv T$ we conclude that $\|T_n-\theta_n\|\Pconv 0$.
	Altogether,  $\WK{A_n}\nconv 0$, i.e. $r_n\cdot R_n(T_n-\theta_n)$ converges to 0 in probability. Thus, by \cref{eq:R_ndef}
	\[r_n\left(\phi(T_n)-\phi(\theta_n)\right)-\bigl(\lim\limits_{n\to\infty}\nabla\phi(\theta_n)\bigr)^t\cdot r_n(T_n-\theta_n)\Pconv 0.\]
	Since $\bigl(\lim\limits_{n\to\infty}\nabla\phi(\theta_n)\bigr)^t\cdot r_n(T_n-\theta_n)$ converges to $\bigl(\lim\limits_{n\to\infty}\nabla\phi(\theta_n)\bigr)^t\cdot T$ in distribution, 
	\[r_n\bigl(\phi(T_n)-\phi(\theta_n)\bigr)\dconv \bigl(\lim\limits_{n\to\infty}\nabla\phi(\theta_n)\bigr)^t\cdot T\]
	which completes the proof.
\end{proof}

\label{ende}
\end{document}